\documentclass[12pt]{amsart}
\usepackage{amssymb,latexsym}
\usepackage{enumerate}
\usepackage[cp1250]{inputenc}

\makeatletter
\@namedef{subjclassname@2010}{%
  \textup{2010} Mathematics Subject Classification}
\makeatother

\newtheorem{thm}{Theorem}[section]
\newtheorem{corollary}[thm]{Corollary}
\newtheorem{lemma}[thm]{Lemma}
\newtheorem{proposition}[thm]{Proposition}

\theoremstyle{definition}
\newtheorem{definition}[thm]{Definition}
\newtheorem{remark}[thm]{Remark}

\numberwithin{equation}{section}

\begin{document}

\baselineskip=17pt

\title[Additivity of the ideal of microscopic sets]{Additivity of the ideal of microscopic sets}

\author[A. Kwela]{Adam Kwela}
\address{Institute of Mathematics, University of Gda\'{n}sk, ul.~Wita Stwosza 57, 80-952 Gda\'{n}sk, Poland}
\email{Adam.Kwela@ug.edu.pl}

\begin{abstract}
A set $M\subset\mathbb{R}$ is microscopic if for each $\varepsilon>0$ there is
a sequence of intervals $(J_n)_{n\in\omega}$ covering $M$ and such that $|J_n|\leq \varepsilon^{n+1}$ for each $n\in\omega$. We show that there is a microscopic set which cannot be covered by a sequence $(J_n)_{n\in\omega}$ with $\{n\in\omega:J_n\neq\emptyset\}$ of lower asymptotic density zero. We prove (in ZFC) that additivity of the ideal of microscopic sets is $\omega_1$. This solves a problem of G. Horbaczewska. Finally, we discuss additivity of some generalizations of this ideal.
\end{abstract}


\maketitle

\section{Introduction}

For $n\in\omega$ we use the identification $n=\{0,1,\ldots,n-1\}$. By $\textrm{card}(A)$ we denote cardinality of a set $A$. For an interval $I\subset\mathbb{R}$ by $|I|$ we denote its length. Given $r\in\mathbb{R}$ and $A\subset\mathbb{R}$ we write $r\cdot A=\{ra:a\in A\}$ and $r+A=\{r+a:a\in A\}$.

We say that a sequence of intervals $(J_n)_{n\in\omega}$ \emph{covers the set} $A\subset\mathbb{R}$ if $A\subset\bigcup_{n\in\omega}J_n$.

\begin{definition}[J. Appell, \cite{Appell}]
A set $M\subset\mathbb{R}$ is called \emph{microscopic} if for each $\varepsilon>0$ there exists a sequence of intervals $(J_n)_{n\in\omega}$ covering $M$ and such that $|J_n|\leq \varepsilon^{n+1}$ for each $n\in\omega$. The family of all microscopic sets will be denoted by $\mathcal{M}$.
\end{definition}

This notion was introduced in 2001 by J. Appell in \cite{Appell}. Deeper studies of microscopic sets were done by J. Appell, E. D'Aniello and M. V\"{a}th in \cite{Appell2}. Since that time, several papers were devoted to this subject, including \cite{Micro1} \cite{Micro2} and \cite{Micro3}. In \cite{Monografia} one can find a summary of the progress made in this area. 

It is easy to see that every microscopic set is contained in some microscopic $\mathbf{G_\delta}$-set, i.e., $\mathcal{M}$ is $\mathbf{G_\delta}$-generated (cf. \cite[Theorem 1.1]{Monografia}). Moreover, $\mathcal{M}$ is strictly smaller than the $\sigma$-ideal of sets of Lebesgue measure zero (cf. \cite{Monografia}). Therefore, many classical theorems stating that some property holds everywhere except a set of Lebesgue measure zero, are being strengthened by showing that actually the set of exceptions can be chosen to be microscopic. For instance, it can be proved that $\mathbb{R}$ can be decomposed into two sets such that one of them is of first category and the second one is microscopic (cf. \cite{Micro1}). 

The aim of this paper is to determine the smallest number of sets from $\mathcal{M}$ union of which is not in $\mathcal{M}$ anymore. For this purpose, we need the notion of asymptotic density of a subset of $\omega$.

Recall that for any $A\subset\omega$ its \emph{upper and lower asymptotic density} are given by the formulas:
$$\overline{d}(A)=\limsup_{j\rightarrow\infty}\frac{\textrm{card}(A\cap (j+1))}{j+1},$$
$$\underline{d}(A)=\liminf_{j\rightarrow\infty}\frac{\textrm{card}(A\cap (j+1))}{j+1}.$$
If $\overline{d}(A)=\underline{d}(A)$, then we say that the set $A$ \emph{is of asymptotic density} $d(A)$ which is equal to this common value.

\begin{definition}
Let $\delta\in [0,1]$. We say that a microscopic set $M\subset\mathbb{R}$ \emph{admits a cover of (lower) asymptotic density} $\delta$ if for every $\varepsilon>0$ there is $D\subset\omega$ with $d(D)\leq\delta$ ($\underline{d}(D)\leq\delta$) and a sequence of intervals $(J_d)_{d\in D}$ which covers $M$ and satisfies $|J_d|\leq \varepsilon^{d+1}$ for all $d\in D$.
\end{definition}

\begin{remark}
\label{rem}
It is easy to see that any microscopic set $M\subset\mathbb{R}$ admits a cover of arbitrarily small positive asymptotic density. Actually, for any $k\in\omega$ and $\varepsilon>0$ one can find a sequence of intervals $(J_d)_{d\in D}$, where $D=(k+1)\cdot(\omega+1)$, which covers $M$ and satisfies $|J_d|\leq \varepsilon^{d+1}$ for each $d\in D$.

Indeed, set any $k\in\omega$ and $\varepsilon>0$. Since $M$ is microscopic, there is a sequence of intervals $(J'_n)_{n\in\omega}$ covering $M$ with $|J'_n|\leq \left(\varepsilon^{k+1}\right)^{n+1}=\varepsilon^{(k+1)(n+1)}$ for each $n\in\omega$. Then it suffices to put $J_{(k+1)(n+1)}=J'_n$ for $n\in\omega$.
\end{remark}

In Section \ref{sec:results} we will show that the above cannot be strengthened, i.e., there is a microscopic set which does not admit a cover of lower asymptotic density zero (cf. Theorem \ref{density}).

From Remark \ref{rem} it easily follows that $\mathcal{M}$ is a $\sigma$-ideal (see \cite{Appell2} or \cite{Monografia} for details). Among studies of $\sigma$-ideals, examination of cardinal coefficients related to them has been of great interest during last decades. This is due to the famous Cicho\'{n}'s diagram which classifies cardinal coefficients of the ideals of null sets and meager sets (cf. \cite{Bartoszynski} and \cite{Fremlin}). 

Recall the definitions of \emph{additivity, covering number, uniformity number and cofinality of an ideal} $\mathcal{I}$ of subsets of $\mathbb{R}$:
\begin{align*}
&\texttt{add}\left(\mathcal{I}\right)=\min\left\{\textrm{card}(\mathcal{A}):\quad\mathcal{A}\subset\mathcal{I}\quad\wedge\quad\bigcup\mathcal{A}\notin\mathcal{I}\right\};\cr
&\texttt{cov}\left(\mathcal{I}\right)=\min\left\{\textrm{card}(\mathcal{A}):\quad\mathcal{A}\subset\mathcal{I}\quad\wedge\quad\bigcup\mathcal{A}=\mathbb{R}\right\};\cr
&\texttt{non}\left(\mathcal{I}\right)=\min\left\{\textrm{card}(A):\quad A\subset\mathbb{R}\quad\wedge\quad A\notin\mathcal{I}\right\};\cr
&\texttt{cof}\left(\mathcal{I}\right)=\min\left\{\textrm{card}(\mathcal{B}):\quad\mathcal{B}\subset\mathcal{I}\quad\wedge\quad\forall_{A\in\mathcal{I}}\exists_{B\in\mathcal{B}}A\subset B\right\}.
\end{align*}

One can easily prove the following inequalities:
$$\texttt{add}(\mathcal{I})\leq\texttt{non}(\mathcal{I})\leq\texttt{cof}(\mathcal{I})\quad\textrm{ and }\quad\texttt{add}(\mathcal{I})\leq\texttt{cov}(\mathcal{I})\leq\texttt{cof}(\mathcal{I}).$$
For more on cardinal coefficients see e.g. \cite{Bartoszynski} or \cite{Fremlin}.

For the ideal of microscopic sets each of those cardinal coefficients lies between $\omega_1$ and $2^{\omega}$ (possibly is equal to one of those two numbers), since $\mathcal{M}$ is a $\sigma$-ideal of subsets of $\mathbb{R}$ containing all singletons and $\mathbf{G_\delta}$-generated. The aim of this paper is to determine additivity of the ideal of microscopic sets. This problem was posed in 2010 by G. Horbaczewska in her talk \emph{Properties of the $\sigma$-ideal of microscopic sets} during XXIV Summer Conference on Real Functions Theory in Stara Lesna, Slovakia. 

Firstly, let us discuss the last three coefficients in the case of microscopic sets. By $\mathcal{N}$ we denote the family of sets of Lebesgue measure zero. Recall that a set $S\subset\mathbb{R}$ is \emph{of strong measure zero} if for each sequence of positive reals $(\varepsilon_n)_{n\in\omega}$ there exists
a sequence of intervals $(J_n)_{n\in\omega}$ covering $S$ and such that $|J_n|\leq \varepsilon_n$ for each $n\in\omega$. The family of sets of strong measure zero will be denoted by $\mathcal{S}$. 

It is well known that both $\mathcal{N}$ and $\mathcal{S}$ are $\sigma$-ideals. One can easily see that $\mathcal{S}\subset\mathcal{M}\subset\mathcal{N}$. In fact, both of these inclusions are proper (cf. \cite{Monografia}).

\begin{remark}
Assume Martin's axiom (cf. \cite{Kunen}). Then $2^{\omega}=\texttt{non}(\mathcal{M})=\texttt{cov}(\mathcal{M})=\texttt{cof}(\mathcal{M})$. Indeed, under Martin's axiom $2^{\omega}=\texttt{add}(\mathcal{N})=\texttt{add}(\mathcal{S})$ (cf. \cite[Theorem 2.1]{smz} and \cite[Theorem 2.21]{Kunen}). Since $\mathcal{S}\subset\mathcal{M}\subset\mathcal{N}$, we also have $\texttt{cov}(\mathcal{N})\leq\texttt{cov}(\mathcal{M})$ and $\texttt{non}(\mathcal{S})\leq\texttt{non}(\mathcal{M})$. Hence,
$$2^{\omega}=\texttt{add}(\mathcal{N})\leq\texttt{cov}(\mathcal{N})\leq\texttt{cov}(\mathcal{M})\leq\texttt{cof}(\mathcal{M})\leq 2^{\omega}$$
and
$$2^{\omega}=\texttt{add}(\mathcal{S})\leq\texttt{non}(\mathcal{S})\leq\texttt{non}(\mathcal{M})\leq 2^{\omega}.$$
\end{remark}

Although $\texttt{non}(\mathcal{M}),\texttt{cov}(\mathcal{M})$ and $\texttt{cof}(\mathcal{M})$ may all be equal to $2^{\omega}$, we will prove in Section \ref{sec:results} that $\texttt{add}(\mathcal{M})$ is always equal to $\omega_1$ (cf. Theorem \ref{add}). 

The paper is organized as follows. In Section \ref{sec:algorithm} we deal with a technical construction which will be helpful in further considerations. In Section \ref{sec:results} we use methods developed in Section \ref{sec:algorithm} to construct a microscopic set which does not admit a cover of lower asymptotic density zero and to prove (in ZFC) that additivity of the ideal of microscopic sets is $\omega_1$. Section \ref{sec:generalizations} is devoted to some generalizations of the ideal of microscopic sets and their additivity.

\section{Spacing algorithm}
\label{sec:algorithm}

\begin{definition}
\label{Y}
Given two sequences of intervals $(I_a)_{a\in A}$ and $(J_d)_{d\in D}$ the set $Y\left((I_a)_{a\in A},(J_d)_{d\in D}\right)$ consists of all $a\in A$ with the following property:
$$\forall_{d\in D}\left(I_a\cap J_d\neq\emptyset\Rightarrow\forall_{\genfrac{}{}{0pt}{}{a'\in A}{a\neq a'}}I_{a'}\cap J_d=\emptyset\right).$$
\end{definition}

In the proofs of Theorems \ref{density} and \ref{add} the following technical lemma will be crucial.

\begin{lemma}[Spacing Algorithm]
\label{lemma}
Let $I$ be an interval of length $\frac{1}{7^m}$ for some $m\in\omega$. Suppose that $A\subset\omega$ is of positive density and $\min A>m$. Then one can define a sequence of intervals $(I_a)_{a\in A}$ with $|I_a|=\frac{1}{7^a}$ and $I_a\subset I$ for all $a\in A$, in such a way that given any $D\subset\omega\setminus m$ and a sequence of intervals $(J_d)_{d\in D}$, with $|J_d|\leq\frac{1}{7^{d+1}}$ for all $d\in D$, for any $s\in\omega$ and $r_0,\ldots,r_s\in (0,1)\setminus\mathbb{Q}$ the set $Z$ consisting of those $a\in Y\left((I_a)_{a\in A},(J_d)_{d\in D}\right)$ which additionally satisfy
$$\forall_{d\in D}\left(I_a\cap J_d\neq\emptyset\Rightarrow\forall_{i\leq s}\forall_{a'\in A}(r_i+I_{a'})\cap J_d=\emptyset\right),$$
is of lower asymptotic density at least $\frac{d(A)}{4}$.
\end{lemma}

\begin{proof}
The proof is divided into five parts. At first, we deal with the construction of the intervals $I_a$ for $a\in A$. Then we focus on preliminary discussion concerning calculation of $\underline{d}(Z)$. The last three parts are devoted to some technical aspects of this calculation.

\textbf{Construction of the intervals $I_a$ for $a\in A$.}

Let $\varepsilon=\frac{1}{7}$. Firstly, we construct auxiliary intervals $K_j^i$ for $i\in\omega$ and $j<4\cdot 3^{i}$. Let $K_j^{0}$ for $j<4$ be such that:
\begin{itemize}
	\item each of them is of length $\varepsilon^{m+1}$;
	\item the distance between each two of them is at least $\varepsilon^{m+1}$;
	\item each of them is contained in $I$;
	\item $\inf K_0^{0}=\inf I$ and $\sup K_1^{0}=\sup I$.
\end{itemize}
Suppose that $K_j^i$ for all $i<k$ and $j<4\cdot 3^{i}$ are defined. Let $K_j^k$ for $j<4\cdot 3^{k}$ be such that:
\begin{itemize}
	\item each of them is of length $\varepsilon^{k+m+1}$;
	\item the distance between each two of them is at least $\varepsilon^{k+m+1}$;
	\item $K_j^k$ is contained in $K_{l}^{k-1}$, where $l=j\mod 3\cdot 3^{k-1}$;
	\item $\inf K_l^{k}=\inf K_{l}^{k-1}$ and $\sup K_{3^k+l}^{k}=\sup K_{l}^{k-1}$.
\end{itemize}

Now we can proceed to the construction of the intervals $I_a$ for $a\in A$. Let $\{a_0,a_1,\ldots\}$ be an increasing enumeration of the set $A$. Define also the family of intervals
$$\mathcal{K}=\{K_j^i:i\in\omega\textrm{ and }3\cdot 3^{i}\leq j<4\cdot 3^{i}\}.$$
Note that for each $K_j^i$ belonging to $\mathcal{K}$ there are no $i'>i$ and $l<4\cdot 3^{i'}$ with $K_l^{i'}$ contained in $K_j^i$. Let $\{K_0,K_1,\ldots\}$ be an enumeration of $\mathcal{K}$ with $|K_i|\geq |K_{i+1}|$. For each $i$ pick $I_{a_i}$ to be any interval of length $\varepsilon^{a_i}$ contained in $K_i$ ($|K_i|\geq \varepsilon^{m+i+1}\geq\varepsilon^{a_i}$ since $\min A>m$).

Observe that for any $i\in\omega$ and $j<3\cdot 3^{i}$ density of the set $\{a\in A:I_a\subset K_j^i\}$ is equal to $d(A)/(3\cdot 3^{i})$.

\textbf{Calculation of $\underline{d}(Z)$.}

We are ready to prove that the intervals $I_a$ for $a\in A$ are as needed. Consider any $s\in\omega$ and $r_0,\ldots,r_s\in (0,1)\setminus\mathbb{Q}$. Set also $D\subset\omega\setminus m$ and a sequence of intervals $(J_d)_{d\in D}$ with $|J_d|\leq\frac{1}{7^{d+1}}$ for all $d\in D$. 

Let $t_n=(3^0+3^1+\ldots+3^n)-1=\frac{3^{n+1}-3}{2}$ and $L_n=\{a_{t_n+1},a_{t_n+1},\ldots,a_{t_{n+1}}\}$ for each $n\in\omega$. The sets $L_n$ are picked in such a way that given $a\in L_n$ the interval $I_{a}$ is contained in $K^{n+1}_j$ for some $3\cdot 3^{n+1}\leq j<4\cdot 3^{n+1}$. 

We will show that for any $\delta>0$ we have
\begin{equation}
	\frac{\textrm{card}(Z\cap (a_{t_{n+1}}+1)\setminus (a_{t_n}+1))}{\textrm{card}(A\cap (a_{t_{n+1}}+1)\setminus (a_{t_n}+1))}>\frac{1}{2}-\delta\label{eq:1}
\end{equation}
for sufficiently large $n$ (equivalently: at least $\frac{1}{2}-\delta$ of all $a\in L_n$ are in $Z$ whenever $n$ is sufficiently large). Once this is done, we conclude that:
\begin{equation}
	\liminf_{n\rightarrow\infty}\frac{\textrm{card}(Z\cap (a_{t_{n}}+1))}{\textrm{card}(A\cap (a_{t_{n}}+1))}\geq\frac{1}{2} \label{eq:2}
\end{equation}
and hence:
$$\liminf_{n\rightarrow\infty}\frac{\textrm{card}(Z\cap (a_{t_{n}}+1))}{a_{t_{n}}+1}\geq\frac{d(A)}{2}.$$

Consider now $a_{t_n}<j<a_{t_{n+1}}$. Recall the definition of $t_n$'s and observe that $\lim_{n\rightarrow\infty}\frac{t_{n+1}-t_n}{2}=t_{n}$. By (\ref{eq:2}) and the fact that $\textrm{card}(A\cap (a_{t_{n}}+1))=t_n$, we get that:
$$\liminf_{n\rightarrow\infty}\frac{\textrm{card}(Z\cap (j+1))}{\textrm{card}(A\cap (j+1))}\geq$$$$\liminf_{n\rightarrow\infty}\frac{\textrm{card}(Z\cap (a_{t_{n}}+1))}{\textrm{card}(A\cap (a_{t_{n}}+1))+\frac{1}{2}\textrm{card}(A\cap ((a_{t_{n+1}}+1)\setminus (a_{t_{n}}+1)))}$$$$=\liminf_{n\rightarrow\infty}\frac{\textrm{card}(Z\cap (a_{t_{n}}+1))}{2\cdot\textrm{card}(A\cap (a_{t_{n}}+1))}.$$
It follows that $\underline{d}(Z)\geq \frac{d(A)}{4}$.

Therefore, it suffices to prove (\ref{eq:1}), i.e., that for any $\delta>0$ at least $\frac{1}{2}-\delta$ of all $a\in L_n$ are in $Z$ whenever $n$ is sufficiently large. Denote $Y=Y\left((I_a)_{a\in A},(J_d)_{d\in D}\right)$ (cf. Definition \ref{Y}) and let $A'$ consist of those $a\in A$ with $(r_i+I_{a'})\cap I_a=\emptyset$ for all $i\leq s$ and $a'\in A$. Set $\delta>0$.

The remaining part of the proof is divided into three steps. At first, we show that at least one half of all $a\in L_n$ is in $Y$ (for all $n\in\omega$). Then we prove that for sufficiently large $n$ at least $1-\delta$ of all $a\in L_n$ is in $A'$. These two steps together show that for sufficiently large $n$ at least $\frac{1}{2}-\delta$ of all $a\in L_n$ is in $Y\cap A'$. Finally, in the last step we conclude that for sufficiently large $n$ at least $\frac{1}{2}-\delta$ of all $a\in L_n$ is in $Z$. 

\textbf{Step 1. The set $Y$.}

Firstly, we will show that for any $n\in\omega$ at least $\frac{1}{2}$ of all $a\in L_n$ is in $Y$. Set $n\in\omega$ and consider the intervals $I_{a}$ for $a\in L_n$. Let $\{d_0,d_1,\ldots\}$ be an increasing enumeration of the set $D$. Observe that $J_{d_0}$ can intersect at most $\frac{1}{3}$ of those intervals. Similarly, $J_{d_0}\cup J_{d_1}$ can intersect at most $\frac{1}{3}+\frac{1}{9}$ of those intervals. Generally, the union of all $J_d$ with $d\in D\cap (n+m+1)$ can intersect at most $\frac{1}{3}+\frac{1}{9}+\ldots<\frac{1}{2}$ of the intervals $I_{a}$ with $a\in L_n$. Each $J_d$ with $d\in D$ and $d\geq n+m+1$ is of length at most $\varepsilon^{n+m+2}$, which is equal to the length of any $K^{n+1}_j$ for $3\cdot 3^{n+1}\leq j<4\cdot 3^{n+1}$. Therefore, each such $J_d$ cannot intersect more than one $I_a$ with $a\in L_n$. Hence, at least $\frac{1}{2}$ of all $a\in L_n$ is in $Y$.

\textbf{Step 2. The set $A'$.}

In this step we show that for sufficiently large $n\in\omega$ at least $1-\delta$ of all $a\in L_n$ is in $A'$.

Since $\sum_{i=0}^\infty\frac{1}{3}\left(\frac{2}{3}\right)^i=1$, there is $k\in\omega$ such that $\left(\sum_{i=0}^k\frac{1}{3}\left(\frac{2}{3}\right)^i\right)^{s+1}>1-\delta$. 

Without loss of generality, we may assume that each $r_i$ is in $(0,\frac{1}{7^m})$ (if some $r_i$ is greater than $\frac{1}{7^m}$, then trivially each $I_a$ is disjoint with the union of $(r_i+I_{a'})_{a'\in A}$). For each $i\leq s$ let $(r_{i,j})_{j\in\omega}\in 7^\omega$ be the unique sequence satisfying $r=\frac{r_{i,0}}{7^{m+1}}+\frac{r_{i,1}}{7^{m+2}}+\ldots$. For each $i\leq s$ let also $(q(i,j))_{j\in\omega}\subset\omega$ be the unique sequence with the following properties:
\begin{itemize}
	\item $q(i,0)$ is minimal with $r_{i,q(i,0)}\neq 0$;
	\item if $j\in\omega$ is such that $r_{i,q(i,j)}$ is odd, then $q(i,j+1)>q(i,j)$ is minimal with $r_{i,q(i,j+1)}\neq 6$;
	\item if $j\in\omega$ is such that $r_{i,q(i,j)}$ is even, then $q(i,j+1)>q(i,j)$ is minimal with $r_{i,q(i,j+1)}\neq 0$.
\end{itemize}
Those sequences are infinite, since $r_i$'s are not in $\mathbb{Q}$.

Pick elements $p(i,j)\in\omega$ for $i\leq s$ and $j\leq k$ such that:
\begin{itemize}
	\item $p(0,j)=q(0,j)$ for each $j\leq k$;
	\item $p(i,j)=q(i,l_i+j)$ for each $0<i\leq s$ and $j\leq k$, where $l_i=\min\{l\in\omega: q(i,l)>p(i-1,k)\}$.
\end{itemize}
Denote $p=p(s,k)$ and let $p'$ be greater than $q(0,k+1)$ and all $q(i,l_i+k+1)$ for $0<i\leq s$. 

In this step we will not need $p'$. The only reason for defining it is to assure in the third step that if $a\in A$ has some required properties, then for all $a'\in A$ and $i\leq s$ we have $(r_i+I_{a'})\cap J_d=\emptyset$ whenever $d\in D$ is such that $I_a\cap J_d\neq\emptyset$.

Set any $n>p$. We will show that at least $1-\delta$ of all $a\in L_n$ is in $A'$.

We need to define an auxiliary set $B\subset L_n$ with $B\subset A'$. Consider the intervals $K^{p(0,0)}_l$ for $l<3^{p(0,0)}$. Each of them is of length $\varepsilon^{p(0,0)+1}$ and therefore is disjoint with the union of $(r_{0}+I_a)_{a\in A}$. Define
$$B^0_0=\left\{a\in L_n: \exists_{l<3^{p(0,0)}}I_a\subset K^{p(0,0)}_l\right\}.$$
Set now any $i\leq s$ and $j\leq k$ with $(i,j)\neq (0,0)$. There are two possible cases.

\textbf{Case 1.} If $q(i,l_i+j-1)$ is even, then $p(i,j)\neq 0$ and each of the intervals $K^{p(i,j)}_l$ for $l<3^{p(i,j)}$ is disjoint with the union of $(r_{i}+I_a)_{a\in A}$ (note that the distance between such $K^{p(i,j)}_l$ and any $x\in\bigcup_{a\in A}(r_i+I+a)$ must be greater than $\frac{1}{7^{p'+m}}$). Define 
$$B^i_j=\left\{a\in L_n: \exists_{l<3^{p(i,j)}}I_a\subset K^{p(i,j)}_l\right\}.$$

\textbf{Case 2.} If $q(i,l_i+j-1)$ is odd, then $p(i,j)\neq 6$ and each of the intervals $K^{p(i,j)}_l$ for $3^{p(i,j)}\leq l<2\cdot 3^{p(i,j)}$ is disjoint with the union of $(r_{i}+I_a)_{a\in A}$  (note that the distance between such $K^{p(i,j)}_l$ and any $x\in\bigcup_{a\in A}(r_i+I+a)$ must be greater than $\frac{1}{7^{p'+m}}$). Define 
$$B^i_j=\left\{a\in L_n: \exists_{3^{p(i,j)}\leq l<2\cdot 3^{p(i,j)}}I_a\subset K^{p(i,j)}_l\right\}.$$

Let $B_i=B^i_0\cup\ldots\cup B^i_k$ and $B=\bigcap_{j\leq s}B_j$. Then $B\subset A'$.

We want to estimate how many of all $a\in L_n$ is in $B$. Denote $\alpha=\sum_{i=0}^k\frac{1}{3}\left(\frac{2}{3}\right)^i$.

Firstly, observe that each $B^i_j$ contains exactly $\frac{1}{3}$ of all $a\in L_n$. What is more, $B^i_0\cup B^i_1$ contains exactly $\frac{1}{3}+\frac{2}{3}\cdot\frac{1}{3}$ of them and, generally, $B_i$ contains exactly $\alpha$ of all $a\in L_n$. Consider now $B_0\cap B_1$. Similarly as above, $B_0\cap B^1_0$ contains exactly $\frac{1}{3}$ of all $a\in B_0$, $B_0\cap (B^1_0\cup B^1_1)$ contains exactly $\frac{1}{3}+\frac{2}{3}\cdot\frac{1}{3}$ of them and, generally, $B_0\cap B_1$ contains exactly $\alpha$ of all $a\in B_0$.

Likewise, we show that for any $i\leq s$ in the set $\bigcap_{j<i}B_j\cap B_i$ there is $\alpha$ of all $a\in \bigcap_{j<i}B_j$. Therefore, $\left(\alpha\right)^s>1-\delta$ of all $a\in L_n$ is in $B\subset A'$.

\textbf{Step 3. The set $Z$.}

By the last two steps we know that at least $\frac{1}{2}-\delta$ of all $a\in L_n$ is in $Y\cap A'$ whenever $n>p$. Observe that the set
$$F=\{a\in Y:\exists_{d\in D\cap (p'+m)}I_a\cap J_d\neq\emptyset\}$$
is finite (actually, of cardinality at most $p'$, by the definition of $Y$ and the fact that $D\subset\omega\setminus m$) and let $N$ be greater than $p$ and $\max\{n\in\omega: \exists_{a\in F}a\in L_n\}$. Pick any $n>N$ and let $B\subset L_n$ be as in the second step. Now we only need to observe that $Y\cap B\subset Z$, i.e., for each $a\in L_n$ with $a\in Y\cap B$ we have
$$\forall_{i\leq s}\forall_{\genfrac{}{}{0pt}{}{a',a''\in A}{a\neq a''}}(r_i+I_{a'})\cap J_d=\emptyset=I_{a''}\cap J_d$$
whenever $d\in D$ is such that $I_a\cap J_d\neq\emptyset$. This finishes the entire proof. 
\end{proof}

\begin{corollary}
\label{cor}
Suppose that $I$ is an interval of length $\frac{1}{7^m}$ for some $m\in\omega$ and $A\subset\omega$ is of positive density with $\min A>m$. Let the sequence of intervals $(I_a)_{a\in A}$ be defined according to Spacing Algorithm. Then for any $D\subset\omega\setminus m$ and a sequence of intervals $(J_d)_{d\in D}$ with $|J_d|\leq\frac{1}{7^{d+1}}$ for all $d\in D$, if $\underline{d}(D)<\frac{d(A)}{4}$, then there is $a\in A$ such that $I_a\cap\bigcup_{d\in D\cap a}J_d$ is empty.
\end{corollary}

\begin{proof}
Denote $Y=Y\left((I_a)_{a\in A},(J_d)_{d\in D}\right)$ and define
$$\delta=\frac{d(A)}{4}-\frac{1}{2}\left(\frac{d(A)}{4}-\underline{d}(D)\right).$$
Firstly, observe that by the Spacing Algorithm $\underline{d}(Y)\geq\frac{d(A)}{4}$ (since $Y$ contains a subset of lower asymptotic density at least $\frac{d(A)}{4}$). Therefore, there is $n_0\in\omega$ such that for every $j>n_0$ we have $\frac{\textrm{card}(Y\cap (j+1))}{j+1}>\delta$. On the other hand, $\underline{d}(D)<\frac{d(A)}{4}$, and hence there is $n_1\in\omega$ such that for every $i>n_1$ one can find $j>i$ with $\frac{\textrm{card}(D\cap (j+1))}{j+1}<\delta$. 

Put $n=\max\{n_0,n_1\}$ and pick any $i>n$. Then there is $j>i$ such that $\frac{\textrm{card}(D\cap (j+1))}{j+1}<\delta$ but $\frac{\textrm{card}(Y\cap (j+1))}{j+1}>\delta$. Hence, $\textrm{card}(D\cap (j+1))<\textrm{card}(Y\cap (j+1))$. By the definition of the set $Y$, each $J_d$ with $d\in D\cap (j+1)$ can intersect at most one $I_a$ with $a\in Y\cap (j+1)$, so there must be some $a\in Y\cap (j+1)$ such that $I_a\cap\bigcup_{d\in D\cap (j+1)}J_d$ is empty. Then also $I_a\cap\bigcup_{d\in D\cap a}J_d$. This finishes the proof.
\end{proof}

\section{Additivity of the ideal of microscopic sets}
\label{sec:results}

In this section we proceed to our main results.

\begin{thm}
\label{density}
There is a bounded microscopic set which does not admit a cover of lower asymptotic density zero.
\end{thm}

\begin{proof}
Let $I=[0,1]$ and $\varepsilon=\frac{1}{7}$. The construction of the required set $X$ is as follows. We inductively define intervals $I^n_j$ for $n\in\omega$ and $j\in 2^n\cdot (\omega+1)$. At the end, we will put $X=\bigcap_{i\in\omega}\bigcup_{j\in 2^i\cdot (\omega+1)}I^i_j$. 

At the first step, apply the Spacing Algorithm for $I$ and $(\omega+1)$ (note that $|I^{-1}_{-1}|=\varepsilon^0$ and $0<\min (\omega+1)$) to get closed intervals $I^0_j$ for $j\in (\omega+1)$ with $|I^0_j|=\varepsilon^{j}$. In the $n$-th step of the induction (for $n>0$) we construct a partition $(A_j^{n-1})_{j\in\omega}$ of the set $2^n\cdot (\omega+1)$ and a sequence of intervals $(I^n_j)_{j\in 2^n\cdot (\omega+1)}$ such that $|I^{n}_j|=\varepsilon^{j}$. The relation between elements of the partitions and the family of intervals is as follows:
$$I^n_j\subset I^{n-1}_{2^{n-1}(k+1)}\quad\Leftrightarrow\quad j\in A^{n-1}_k.$$
So suppose now that $A^i_k$ and $I^{i+1}_j$ are defined for all $i<m$, $k\in\omega$ and $j\in 2^{i+1}\cdot (\omega+1)$. Let $A^{m}_j=2^{m+j+1}+2^{m+j+2}\cdot\omega$ for all $j\in\omega$. Then $(A^{m}_j)_{j\in\omega}$ is a partition of $2^{m+1}\cdot (\omega+1)$ into sets of positive density. For each $n\in\omega$ apply the Spacing Algorithm for $I^{m}_{2^{m}(n+1)}$ and $A^{m}_n$ (note that $2^{m}(n+1)<2^{m+n+1}=\min A^{m}_n$) to get closed intervals $I^{m+1}_j$ for $j\in 2^{m+1}\cdot (\omega+1)$ with $|I^{m+1}_j|=\varepsilon^{j}$ and
$$I^{m+1}_j\subset I^{m}_{2^{m}(n+1)}\quad\Leftrightarrow\quad j\in A^{m}_n.$$

Finally, define the sets 
$$X_i=\bigcup_{j\in 2^i\cdot (\omega+1)}I^i_j\textrm{ and }X=\bigcap_{i\in\omega}X_i.$$ 
Then $X$ is a bounded microscopic set. Indeed, given $\varepsilon'>0$ one can find $m>1$ with $\varepsilon^{2^m}<\varepsilon'$. Then it suffices to note that the sequence of intervals $(I^m_{2^m(j+1)})_{j\in \omega}$ covers $X_m$ (and hence the whole set $X$) and 
$$|I^m_{2^m(j+1)}|=\varepsilon^{2^m(j+1)}<(\varepsilon')^{(j+1)}.$$

Now let $D\subset\omega$ be of lower asymptotic density zero and $(J_d)_{d\in D}$ be a sequence of intervals such that $|J_d|\leq\varepsilon^{d+1}$ for all $d\in D$. We will show that $(J_d)_{d\in D}$ cannot cover the set $X$ by constructing an increasing sequence $(j_n)_{n\in\omega}\subset\omega$ with $I^n_{j_n}\supset I^{n+1}_{j_{n+1}}$ and $I^n_{j_n}\cap\bigcup_{d\in D\cap j_n}J_d=\emptyset$. Then $\bigcap_{n\in\omega}I^n_{j_n}$ will define a point from $X$ which does not belong to $\bigcup_{d\in D}J_d$.

The construction of the sequence $(j_n)_{n\in\omega}$ is as follows. By Corollary \ref{cor} (applied to $I$, $(\omega+1)$ and the sequence of intervals $(J_d)_{d\in D}$), there is $j_0\in (\omega+1)$ such that $I^0_{j_0}\cap\bigcup_{d\in D\cap j_0}J_d=\emptyset$. Assume now that $j_i$ for $i\leq n$ are as needed. Again, by Corollary \ref{cor} (applied to $I^n_{j_n}$, $A^n_{j_n/2^n-1}$ and the sequence of intervals $(J_d)_{d\in D\setminus j_n}$), we can find $j_{n+1}\in A^n_{j_n/2^n-1}$ (hence $I^n_{j_n}\supset I^{n+1}_{j_{n+1}}$) with $I^{n+1}_{j_{n+1}}\cap\bigcup_{d\in D\cap (j_{n+1}\setminus j_n)}J_d=\emptyset$. Then also $I^{n+1}_{j_{n+1}}\cap\bigcup_{d\in D\cap j_{n+1}}J_d=\emptyset$, by $I^n_{j_n}\supset I^{n+1}_{j_{n+1}}$ and the induction assumption. This ends the construction and the entire proof.
\end{proof}

We are ready to prove the main theorem of this paper.  

\begin{thm}
\label{add}
Additivity of the ideal of microscopic sets is equal to $\omega_1$.
\end{thm}

\begin{proof}
Recall that $\texttt{add}\left(\mathcal{M}\right)\geq\omega_1$ (cf. Remark \ref{rem} and the discussion below it). Therefore, it suffices to prove that there is a family of cardinality $\omega_1$ consisting of microscopic sets and such that its union is not microscopic. 

Define the set
$$T=\left\{(i,j):\textrm{ }i\in\omega\textrm{ and }j\in 2^i\cdot (\omega+1)\right\}\cup\{(-1,-1)\}$$ 
and put $I^{-1}_{-1}=[0,1]$. Let $X$ and $I^i_j$ for $(i,j)\in T$ be as in the proof of Theorem \ref{density} and pick a family $\{r_\alpha:\alpha<\omega_1\}\subset (0,1)$ such that $r_\alpha-r_\beta\notin\mathbb{Q}$ whenever $\alpha\neq\beta$. Define $X_\alpha=r_\alpha+X$ for all $\alpha<\omega_1$. Clearly, each $X_\alpha$ is microscopic. We will show that $\bigcup_{\alpha<\omega_1}X_\alpha$ is not microscopic.

Set $\varepsilon=\frac{1}{7}$ and any sequence of intervals $(J_n)_{n\in \omega}$ such that $|J_n|\leq\varepsilon^{n+1}$ for all $n\in\omega$. Assume that $(J_n)_{n\in \omega}$ covers $\bigcup_{\alpha<\omega_1}X_\alpha$.

Consider the case that there is $\alpha<\omega_1$ such that for any pair $(n,m)\in T$ if $(r_\alpha+I_m^n)\cap\bigcup_{k<m}J_k=\emptyset$, then one can find $l\in 2^{n+1}\cdot (\omega+1)$ such that $(r_\alpha+I_l^{n+1})\subset (r_\alpha+I_m^n)$ and $(r_\alpha+I_l^{n+1})\cap\bigcup_{k<l}J_k=\emptyset$. This condition allows us to construct an increasing sequence $(m_n)_{n\in\omega}$ such that $(r_\alpha+I_{m_{n+1}}^{n+1})\subset (r_\alpha+I_{m_n}^{n})$ and $(r_\alpha+I_{m_n}^{n})\cap\bigcup_{k<m_n}J_k=\emptyset$ for all $n\in\omega$. Hence, the intersection $\bigcap_{n\in\omega}(r_\alpha+I_{m_n}^{n})$ defines a point from $X_\alpha$ (and hence from $\bigcup_{\alpha<\omega_1}X_\alpha$) which is disjoint with the union $\bigcup_{k\in\omega}J_k$. 

Therefore, we can assume from now on that for any $\alpha<\omega_1$ there is a pair $(n_\alpha,m_\alpha)\in T$ such that $(r_\alpha+I_{m_\alpha}^{n_\alpha})\cap\bigcup_{k<m_\alpha}J_k=\emptyset$ but $(r_\alpha+I_{j}^{n_\alpha+1})\cap\bigcup_{k<j}J_k\neq\emptyset$ whenever $j\in 2^{n_\alpha+1}\cdot (\omega+1)$ is such that $(r_\alpha+I_{j}^{n_\alpha+1})\subset (r_\alpha+I_{m_\alpha}^{n_\alpha})$ (note that trivially $(r_\alpha+I_{-1}^{-1})\cap\bigcup_{k<-1}J_k=\emptyset$ for all $\alpha<\omega_1$). There are only countably many possible choices for $(n_\alpha,m_\alpha)$, so one can find  an uncountable $F\subset\omega_1$ and a pair $(n,m)\in T$ such that $(n,m)=(n_\alpha,m_\alpha)$ for all $\alpha\in F$.

Define the set $A=\left\{a\in 2^{n+1}\cdot (\omega+1):I^{n+1}_a\subset I^n_m\right\}$ (note that $A=A^m_{m/2^n-1}$ in the notation from the proof of Theorem \ref{density}). By the construction of the set $X$ we have $d(A)>0$. Let $s\in (\omega+1)$ be such that $\frac{1}{s}\leq\frac{d(A)}{4}$ and pick $\alpha_0,\ldots,\alpha_s\in F$ with $r_{\alpha_0}<r_{\alpha_1}<\ldots<r_{\alpha_s}$. For each $i\leq s$ let $Y_i\subset A$ be the set $Y\left((r_{\alpha_i}+I^{n+1}_{a})_{a\in A},(J_k)_{k\in\omega}\right)$ (cf. Definition \ref{Y}). Let also $Z_i$, for $i\leq s$, be the set of those $a\in Y_i$ which have the property that given any $k\in\omega$ if $(r_{\alpha_i}+I^{n+1}_{a})\cap J_k\neq\emptyset$, then there are no $i<j\leq s$ and $a'\in A$ such that $(r_{\alpha_j}+I^{n+1}_{a'})\cap J_k\neq\emptyset$ (hence $Z_s=Y_s$).

By the Spacing Algorithm, for each $i\leq s$ the set $Z_i$ has lower asymptotic density at least $\frac{d(A)}{4}\geq\frac{1}{s}$. Define
$$Z'_i=\left\{k\in\omega: \exists_{a\in Z_i}(r_{\alpha_i}+I^{n+1}_{a})\cap J_k\neq\emptyset\right\}$$
for all $i\leq s$. Those sets also have lower asymptotic density at least $\frac{1}{s}$. Indeed, set any $i\leq s$ and consider a bijection $\phi$ between $Z_i$ and $Z'_i$ such that $\phi(a)$ is equal to $k\in Z'_i$ if $(r_{\alpha_i}+I^{n+1}_{a})\cap J_k\neq\emptyset$ for $a\in Z_i$. This function is well defined, since $k$ with the above property is unique by the definition of $Y_i$. Observe that $(r_{\alpha_i}+I_{a}^{n+1})\cap\bigcup_{k<a}J_k\neq\emptyset$ for all $a\in A$ by $\alpha_i\in F$ and the definition of $(n,m)$. Therefore $\phi(a)\leq a$ for all $a\in Z_i$. It follows that $\underline{d}(Z'_i)\geq\underline{d}(Z_i)\geq\frac{1}{s}$.

Moreover, $Z'_i\cap Z'_j=\emptyset$ for $i<j\leq s$. Indeed, if $k\in Z'_i$, then $(r_{\alpha_i}+I^{n+1}_{a})\cap J_k\neq\emptyset$ for some $a\in Z_i$, and hence (by the definition of $Z_i$) there is no $a'\in A$ such that $(r_{\alpha_j}+I^{n+1}_{a'})\cap J_k\neq\emptyset$, which means that $k\notin Z'_j$. 

Therefore, $\{Z'_0,\ldots,Z'_s\}$ is a family of $s+1$ pairwise disjoint subsets of $\omega$, each of which is of lower asymptotic density at least $\frac{1}{s}$. A contradiction. Hence, $(J_n)_{n\in \omega}$ cannot cover the set $\bigcup_{\alpha<\omega_1}X_\alpha$.
\end{proof}

\section{Some generalizations of the ideal of microscopic sets}
\label{sec:generalizations}

In this section we investigate additivity of two ideals closely related to $\mathcal{M}$. 

\begin{definition}
A set $M\subset\mathbb{R}$ is in $\mathcal{M}_{\ln}$ if for each $\varepsilon>0$ there exists a sequence of intervals $(J_n)_{n\in\omega}$ covering $M$ and such that $|J_n|\leq \varepsilon^{\ln (n+2)}$ for each $n\in\omega$.
\end{definition}

\begin{definition}
A set $M\subset\mathbb{R}$ is in $\mathcal{M}'$ if for each $\varepsilon>0$ there exists $D\subset\omega$ of asymptotic density zero and a sequence of intervals $(J_n)_{n\in D}$ such that $M\subset\bigcup_{n\in D}J_n$ and $|J_n|\leq \varepsilon^{n+1}$ for each $n\in D$.
\end{definition}

Recently, Horbaczewska in \cite{Horbaczewska} defined the so-called $(f_n)_{n\in\omega}$-microscopic sets. This concept was deeply studied in \cite{nano}. Let us point out that in the terminology of \cite{Horbaczewska}, $\mathcal{M}_{\ln}$ is the family of all $(x^{\ln (n+2)})_{n\in\omega}$-microscopic sets.

Observe that $\mathcal{S}\subset\mathcal{M}'\subset\mathcal{M}\subset\mathcal{M}_{\ln}$. In fact, all inclusions are proper. One can easily construct a compact microscopic set of cardinality $2^\omega$, which shows that $\mathcal{S}\neq\mathcal{M}'$. Theorem \ref{density} gives us an example of a microscopic set not belonging to $\mathcal{M}'$. Finally, $\mathcal{M}\neq\mathcal{M}_{\ln}$ will follow from the fact that $\mathcal{M}_{\ln}$ has additivity $2^\omega$ under Martin's axiom (cf. Proposition \ref{ln}).

The following lemma will be useful in our further considerations.

\begin{lemma}
\label{ln-lemma}
Set $M\in\mathcal{M}_{\ln}$ and $\varepsilon\in (0,1)$. Suppose that $(J_d)_{d\in D}$ is such that $d(D)=0$ and $|J_d|\leq\varepsilon^{\ln (d+2)}$ for all $d\in D$. Then there are $E\subset\omega$ disjoint with $D$ and of asymptotic density zero and a sequence of intervals $(J_e)_{e\in E}$ covering $M$ and such that $|J_e|\leq \varepsilon^{\ln (e+2)}$ for each $e\in E$. 
\end{lemma}

\begin{proof}
Take any $(J_d)_{d\in D}$ such that $d(D)=0$ and $|J_d|\leq\varepsilon^{\ln (d+2)}$ for all $d\in D$. Since $d(D)=0$, there is $k\in\omega$ such that $\frac{\textrm{card}(D\cap (j+1))}{j+1}\leq\frac{1}{4}$ for all $j>k$. Find $m\in\omega$ such that $2^m>k$ and $m\geq 2$. We inductively pick a sequence $(t_i)_{i\in\omega}$ of pairwise distinct elements of $\omega\setminus D$ satisfying $\frac{1}{2}(i+2)^m\leq t_i+2\leq (i+2)^m$. 

The construction is as follows. Let $t_0\in\omega\setminus D$ be maximal such that $t_0+2\leq 2^m$. Note that at most one in four of all $n\leq 2^m$ is in $D$, hence $t_0+2\geq \frac{1}{2}2^m$. Assume now that $t_0,\ldots,t_{i-1}$ are constructed. Pick $t_i\in \omega\setminus (D\cup\{t_0,\ldots,t_{i-1}\})$ to be maximal such that $t_i+2\leq (i+2)^m$. Note that at most one in four of all $n\leq (i+2)^m$ is in $D$. Moreover, until this moment we have picked only $i$ numbers from $\omega\setminus D$ and $\frac{i}{(i+2)^m}<\frac{1}{4}$, so less than one in four of all $n\leq (i+2)^m$ is one of the $t_j$'s for $j<i$. Therefore, $\frac{1}{2}(i+2)^m\leq t_i+2$. 

Define $E=\{t_i:i\in\omega\}$. Obviously, $D\cap E=\emptyset$. What is more, $d(E)=0$. Indeed, given any $j\in\omega$, the number of elements of the set $E\cap (j+1)$ is bounded above by $(2(j+2))^{\frac{1}{m}}-1$, since $j<\frac{1}{2}(i+2)^m-2\leq t_i$ whenever $i>(2(j+2))^{\frac{1}{m}}-2$. Now it suffices to observe that:
$$\frac{\textrm{card}(E\cap (j+1))}{j+1}\leq\frac{(2(j+2))^{\frac{1}{m}}-1}{j+1}\rightarrow 0.$$

Since the set $M$ is in $\mathcal{M}_{\ln}$, there is a sequence of intervals $(I_n)_{n\in\omega}$ covering $M$ and such that $|I_n|\leq (\varepsilon^m)^{\ln (n+2)}=\varepsilon^{\ln (n+2)^m}$. Let $J_{t_n}=I_n$ for all $n\in\omega$ and note that for all $n\in\omega$ we have $|I_n|\leq\varepsilon^{\ln (t_n+2)}$, since $t_n+2\leq (n+2)^m$ and $\varepsilon\in (0,1)$.  
\end{proof}

\begin{proposition}
Both $\mathcal{M}_{\ln}$ and $\mathcal{M}'$ are $\sigma$-ideals. 
\end{proposition}

\begin{proof}
Firstly, assume that $(M_k)_{k\in\omega}\subset\mathcal{M}'$ and set any $\varepsilon>0$. Similarly as in Remark \ref{rem}, for each $k\in\omega$ we can find a sequence of intervals $(J_d)_{d\in D_k}$, where $D_k\subset 2^{k+1}\cdot(\omega+1)$, which covers $M_k$ and satisfies $|J_d|\leq \varepsilon^{d+1}$ for each $d\in D_k$. Let $D'_k=D_k-2^k$ and note that $(D'_k)_{k\in\omega}$ is a family of pairwise disjoint subsets of $\omega$. Let also $J'_d=J_{d+2^k}$ whenever $d\in D'_k$ and define $D=\bigcup_{k\in\omega}D'_k$. Then $(J'_n)_{n\in D}$ covers $\bigcup_{k\in\omega}M_k$ and $|J'_n|\leq\varepsilon^{n+1}$ for each $n\in\omega$. 

We need to show that $D$ is of asymptotic density zero. Set any $\delta>0$. There is $m\in\omega$ such that $D^0=\bigcup_{k>m}(2^{k+1}\cdot(\omega+1)-2^k)$ has asymptotic density less than $\frac{\delta}{3}$. Hence, there is $j_0\in\omega$ such that for all $j>j_0$ we have $\frac{\textrm{card}(D^0\cap (j+1))}{j+1}<\frac{\delta}{2}$. Denote $D^1=\bigcup_{k\leq m}D'_k$ and note that this set is of asymptotic density zero. Hence, there also is $j_1\in\omega$ such that for all $j>j_0$ we have $\frac{\textrm{card}(D^1\cap (j+1))}{j+1}<\frac{\delta}{2}$. Since $D\subset D^0\cup D^1$, we have:
$$\frac{\textrm{card}(D\cap (j+1))}{j+1}\leq\frac{\textrm{card}(D^0\cup D^1\cap (j+1))}{j+1}<\delta$$
whenever $j>\max\{j_0,j_1\}$.

Assume now that $(M_k)_{k\in\omega}\subset\mathcal{M}_{\ln}$ and set any $\varepsilon>0$. By Lemma \ref{ln-lemma} (applied to $D=\emptyset$), there are $E_0\subset\omega$ of asymptotic density zero and a sequence of intervals $(J_e)_{e\in E_0}$ covering $M_0$ and such that $|J_e|\leq \varepsilon^{\ln (e+2)}$ for each $e\in E_0$. However, by Lemma \ref{ln-lemma} there also are $E_1\subset\omega$ disjoint with $E_0$ of asymptotic density zero and a sequence of intervals $(J_e)_{e\in E_1}$ covering $M_1$ and such that $|J_e|\leq \varepsilon^{\ln (e+2)}$ for each $e\in E_1$. In this way we inductively construct a family $(E_n)_{n\in\omega}$ of pairwise disjoint subsets of $\omega$ and a sequence of intervals $(J_e)_{e\in E}$, where $E=\bigcup_{n\in\omega}E_n$, covering $\bigcup_{n\in\omega}M_n$ and such that $|J_e|\leq \varepsilon^{\ln (e+2)}$ for each $e\in E$. 
\end{proof}

Now we will calculate additivity of the ideals $\mathcal{M}_{\ln}$ and $\mathcal{M}'$ under Martin's axiom.

\begin{proposition}
\label{ln}
Assume MA$_\kappa$. If $\mathcal{F}\subset\mathcal{M}_{\ln}$ is a family of cardinality $\kappa$, then $\bigcup\mathcal{F}\in\mathcal{M}_{\ln}$. Therefore, $\texttt{add}\left(\mathcal{M}_{\ln}\right)=2^\omega$ under Martin's axiom.
\end{proposition}

The proof is an adaptation of the proof of \cite[Theorem 2.1]{smz}. Therefore, we omit some details and focus only on the modified parts.

\begin{proof}
Let $\mathcal{F}=\{M_\alpha:\alpha<\kappa\}$ and $\mathcal{B}$ be the family of all open intervals with rational endpoints. Notice that $\mathcal{B}$ is countable. Denote $M=\bigcup_{\alpha<\kappa}M_\alpha$ and take any $\varepsilon\in (0,1)$. Let 
$$\mathcal{P}=\left\{(J_d)_{d\in D}:d(D)=0\textrm{ and }\forall_{d\in D}\left(J_d\in\mathcal{B}\textrm{ and }|J_d|\leq\varepsilon^{\ln (d+2)}\right)\right\}$$
and define the relation $\prec$ on $\mathcal{P}$ by:
$$(J_d)_{d\in D}\prec (J'_d)_{d\in D'}\Leftrightarrow \bigcup_{d\in D}J_d\supset\bigcup_{d\in D'}J'_d.$$
Then $(\mathcal{P},\prec)$ is a partial order which is c.c.c. (for details see \cite{smz}). For all $\alpha<\kappa$ define also
$$\mathcal{D}_\alpha=\left\{(J_n)_{n\in D}\in\mathcal{P}:M_\alpha\subset\bigcup_{n\in D}J_n\right\}.$$
We want to prove that these sets are dense.

Take any $\alpha<\kappa$. We will show that $\mathcal{D}_\alpha$ is dense. Suppose that $(J_d)_{d\in D}\in\mathcal{P}$. By Lemma \ref{ln-lemma}, there is $E\subset\omega$ disjoint with $D$ and of asymptotic density zero and a sequence of intervals $(J_e)_{e\in E}$ covering $M$ and such that $|J_e|\leq \varepsilon^{\ln (e+2)}$ for each $e\in E$. Observe that the sequence $(J_n)_{n\in D\cup E}$ is in $\mathcal{D}_\alpha$ and $(J_n)_{n\in D\cup E}\prec (J_n)_{n\in D}$. Therefore, the set $\mathcal{D}_\alpha$ is dense.

By MA$_\kappa$, there is a filter $\mathcal{G}$ in $\mathcal{P}$ intersecting each $\mathcal{D}_\alpha$ for $\alpha<\kappa$. Let also $I_0,I_1,\ldots$ list all the intervals $J$ such that there are $(J_d)_{d\in D}\in\mathcal{G}$ and $d\in D$ with $J=J_d$ (note here that each $J_d$ is in $\mathcal{B}$ and recall that $\mathcal{B}$ is countable). Then the union $\bigcup_{n\in \omega}I_n$ covers the set $M$, since each $M_\alpha$ is contained in some $\bigcup_{d\in D}J_d$ with $(J_d)_{d\in D}\in\mathcal{G}$. Moreover, $|I_n|\leq \varepsilon^{\ln (n+2)}$ for all $n\in\omega$ (for details see \cite{smz}). Hence, the set $M$ is in $\mathcal{M}_{\ln}$.
\end{proof}

\begin{proposition}
Assume MA$_\kappa$. If $\mathcal{F}\subset\mathcal{M}'$ is a family of cardinality $\kappa$, then $\bigcup\mathcal{F}\in\mathcal{M}'$. Therefore, $\texttt{add}\left(\mathcal{M}'\right)=2^\omega$ under Martin's axiom.
\end{proposition}

This proof also is based on the proof of \cite[Theorem 2.1]{smz} and is very similar to the proof of Proposition \ref{ln}.

\begin{proof}
Let $\mathcal{F}=\{M_\alpha:\alpha<\kappa\}$ and denote $M=\bigcup_{\alpha<\kappa}M_\alpha$. Similarly as in the proof of Proposition \ref{ln}, let $\mathcal{B}$ be the family of all open intervals with rational endpoints and take any $\varepsilon\in (0,1)$. Let 
$$\mathcal{P}=\left\{(J_d)_{d\in D}:d(D)=0\textrm{ and }\forall_{d\in D}\left(J_d\in\mathcal{B}\textrm{ and }|J_d|\leq\varepsilon^{d+1}\right)\right\}$$
and define the relation $\prec$ on $\mathcal{P}$ by:
$$(J_d)_{d\in D}\prec (J'_d)_{d\in D'}\Leftrightarrow \bigcup_{d\in D}J_d\supset\bigcup_{d\in D'}J'_d.$$
Then, as in the proof of Proposition \ref{ln}, $(\mathcal{P},\prec)$ is a partial order which is c.c.c. For all $\alpha<\kappa$ define also
$$\mathcal{D}_\alpha=\left\{(J_d)_{d\in D}\in\mathcal{P}:M_\alpha\subset\bigcup_{d\in D}J_d\right\}.$$
Now it suffices to prove that these sets are dense. 

Take any $\alpha<\kappa$. We want to show that $\mathcal{D}_\alpha$ is dense. Suppose that $(J_d)_{d\in D}\in\mathcal{P}$. Since $d(D)=0$, there is $m\in\omega$ such that $\frac{\textrm{card}(D\cap (j+1))}{j+1}<\frac{1}{4}$ for all $j\geq m$. We can additionally assume that $m\geq 4$ and $m$ is even. Since the set $M_\alpha$ is in $\mathcal{M}'$, there is $E\subset\omega$ of density zero and a sequence of intervals $(I_e)_{e\in E}$ covering $M_\alpha$ and such that $|I_e|\leq\left(\varepsilon^m\right)^{e+1}=\varepsilon^{m(e+1)}$ for all $e\in E$. Let $\{e_0,e_1,\ldots\}$ be an increasing enumeration of the set $E$. We inductively pick a sequence $(t_i)_{i\in\omega}$ of pairwise distinct elements of $\omega\setminus D$ satisfying $\frac{1}{2}m(e_i+1)\leq t_i+1\leq m(e_i+1)$.

The construction is as follows. Let $t_0\in\omega\setminus D$ be maximal such that $t_0+1\leq m(e_0+1)$. Note that at most one in four of all $n\leq m(e_0+1)$ is in $D$, and hence $t_0+1\geq \frac{1}{2}m(e_0+1)$. Assume now that $t_0,\ldots,t_{i-1}$ are constructed. Pick $t_i\in \omega\setminus (D\cup\{t_0,\ldots,t_{i-1}\})$ to be maximal such that $t_i+1\leq m(e_i+1)$. Note that at most one in four of all $n\leq m(e_i+1)$ is in $D$. Moreover, until this moment we have picked only $i$ numbers from $\omega\setminus D$ and by the fact that $m\geq 4$, we have $\frac{i}{m(e_i+1)}\leq\frac{e_i+1}{m(e_i+1)}<\frac{1}{4}$. Hence, less than one in four of all $n\leq m(e_i+1)$ is one of the $t_j$'s for $j<i$. Therefore, $\frac{1}{2}m(e_i+1)\leq t_i+1$.

Define $F=\{t_i:i\in \omega\}$. Obviously, $D\cap F=\emptyset$. What is more, $d(F)=0$. Indeed, given any $j\in\omega$, the number of elements of the set $F\cap (j+1)$ is bounded above by the cardinality of the set $\{i\in\omega:e_i\leq\frac{2(j+1)}{m}-1\}=E\cap\frac{2(j+1)}{m}$, since $j+1<\frac{1}{2}m(e_i+1)\leq t_i+1$ whenever $e_i>\frac{2(j+1)}{m}-1$. Now it suffices to observe that:
$$\frac{\textrm{card}(F\cap (j+1))}{j+1}\leq\frac{\textrm{card}(E\cap\frac{2(j+1)}{m})}{j+1}\leq\frac{\textrm{card}(E\cap (j+1))}{j+1}\rightarrow 0,$$
since $E$ is of asymptotic density zero.

Let $J_{t_i}=I_{e_i}$ for all $i\in \omega$ and note that we have $|J_{t_i}|\leq\varepsilon^{(t_i+1)}$, since $t_i+1\leq m(e_i+1)$. Then the sequence $(J_{n})_{n\in D\cup F}$ is in $\mathcal{D}_\alpha$ and $(J_n)_{n\in D\cup F}\prec(J_n)_{n\in D}$. Therefore, the set $\mathcal{D}_\alpha$ is dense.

The rest of the proof is similar to the proof of Proposition \ref{ln}.
\end{proof}


\begin{thebibliography}{HD}

\normalsize
\baselineskip=17pt

\bibitem[1]{Appell}   
J. Appell,
\emph{Insiemi ed operatori "piccoli" in analisi funzionale}, 
Rend. Instit. Mat. Univ. Trieste, {\bf 33}, (2001), 127--199.

\bibitem[2]{Appell2}   
J. Appell, E. D'Aniello, M. V\"{a}th,
\emph{Some remark on small sets}, 
Ric. Mat., {\bf 50}, (2001), 255--274.

\bibitem[3]{Bartoszynski}   
T. Bartoszy\'{n}ski,
\emph{Invariants of Measure and Category},
In: Handbook of Set Theory, Springer-Verlag, (2010), 491--555.

\bibitem[4]{smz}   
T. J. Carlson,
\emph{Strong measure zero and strongly meager sets}, 
Proc. Am. Math. Soc., {\bf 118}, (1993), 577--586.

\bibitem[5]{nano}   
K. Czudek, A. Kwela, N. Mro\.{z}ek, W. Wo{\l}oszyn,
\emph{Some remarks on microscopic sets}, 
preprint.

\bibitem[6]{Fremlin}   
D. H. Fremlin,
\emph{Cichon's diagram}, 
1984, presented at the S\'{e}minaire Initiation \`{a} l'Analyse, G. Choquet, M. Rogalski, J. Saint Raymond, at the Universit\'{e} Pierre et Marie Curie, Paris, 23e ann\'{e}e.

\bibitem[7]{Horbaczewska}   
G. Horbaczewska,
\emph{Microscopic sets with respect to sequences of functions}, 
Tatra Mountains Math. Publ., {\bf 58}, (2014), 137--144.

\bibitem[8]{Monografia}   
G. Horbaczewska, A. Karasi\'{n}ska, E. Wagner-Bojakowska,
\emph{Properties of the $\sigma$-ideal of microscopic sets}, 
In: Traditional and present-day topics in real analysis, \L\'{o}d\'{z} University Press, (2013), 325--343.

\bibitem[9]{Micro1}
A. Karasi\'{n}ska, W. Poreda, E. Wagner-Bojakowska, 
\emph{Duality Principle for microscopic sets},
In: Real Functions, Density Topology and Related Topics, \L\'{o}d\'{z} University Press,
(2011), 83--87.

\bibitem[10]{Micro2}
A. Karasi\'{n}ska, E. Wagner-Bojakowska, 
\emph{Nowhere monotone functions and microscopic sets},
Acta Math. Hungar., {\bf 120}, (2008), 235--248.

\bibitem[11]{Micro3}
A. Karasi\'{n}ska, E. Wagner-Bojakowska, 
\emph{Homeomorphisms of linear and planar sets of the first category into microscopic sets}, Topology Appl., {\bf 159}, (2012), 1894--1898.

\bibitem[12]{Kunen} 
K. Kunen,
\emph{Set Theory: An Introduction to Independence Proofs}, North-Holland Publishing Co., Amsterdam, 1980.

\end{thebibliography}
\end{document}